\documentclass[11pt, oneside]{article}   	
\usepackage{geometry}                		
\usepackage{graphicx}				
								
\usepackage[all]{xy}
\CompileMatrices

\usepackage{amssymb}
\usepackage{amsmath}
\usepackage{stmaryrd}
\usepackage{amsthm}
\usepackage{tikz-cd}
\usepackage{extarrows}
\usepackage[mathscr]{euscript}
\usepackage{mathrsfs} 
\usepackage{verbatim}
\usepackage[OT2,T1]{fontenc}
\DeclareSymbolFont{cyrletters}{OT2}{wncyr}{m}{n}
\DeclareMathSymbol{\Sha}{\mathalpha}{cyrletters}{"58}
\DeclareMathSymbol{\Che}{\mathalpha}{cyrletters}{"51}

\usepackage{calligra}
\usepackage{mathrsfs}

\newcommand{\Ga}{{\mathbf{G}}_{\rm{a}}}
\newcommand{\Gm}{{\mathbf{G}}_{\rm{m}}}
\DeclareMathOperator{\red}{red}
\DeclareMathOperator{\perf}{perf}

\DeclareMathOperator{\Pic}{Pic}

\DeclareMathOperator{\Hom}{Hom}

\DeclareMathOperator{\et}{\acute{e}t}

\newcommand*{\Z}{\ensuremath{\mathbf{Z}}}                        
\newcommand*{\address}{Einstein Institute of Mathematics, The Hebrew University of Jerusalem, Edmond J. Safra Campus, 91904, Jerusalem, Israel}
\newcommand*{\email}{zev.rosengarten@mail.huji.ac.il}

\usepackage{rotating}

\usepackage{bm}

\newtheorem{theorem}{Theorem}

\newtheorem{lemma}[theorem]{Lemma}

\theoremstyle{definition}

  \theoremstyle{remark}

\theoremstyle{remark}

\usepackage[OT2,T1]{fontenc}

\tikzset{commutative diagrams/.cd,
mysymbol/.style={start anchor=center,end anchor=center,draw=none}
}

\usepackage{stmaryrd}
\usepackage{hyperref}

\title{\textbf{PICARD GROUPS OF ALGEBRAIC GROUPS AND AN AFFINENESS CRITERION}}
\author{Zev Rosengarten }

\date{}
\begin{document}
\maketitle

\begin{abstract}
Nous prouvons qu’un groupe algébrique sur un corps $k$ est affine 
si et seulement si son groupe de Picard est de torsion, et que dans
ce cas, le groupe de Picard est fini si $k$ est parfait, et le produit 
d’un groupe  fini  d'ordre premier à $p$ par un $p$-groupe d'exposant
fini lorsque $k$ est imparfait de caractéristique $p$.
\\

We prove that an algebraic group over a field $k$ is affine precisely when its Picard group is torsion, and show that in this case the Picard group is finite when $k$ is perfect, and the product of a finite group of order prime to $p$ and a $p$-primary group of finite exponent when $k$ is imperfect of characteristic $p$.
\end{abstract}

The main purpose of this short paper is to prove the following two results on algebraic groups over arbitrary fields.

\begin{theorem}
\label{torsionpic=affine}
Let $G$ be a finite type group scheme over a field $k$. Then $G$ is affine if and only if $\Pic(G)$ is torsion.
\end{theorem}

\begin{theorem}
\label{pic(affine)=finitetimespprimary}
Let $G$ be an affine group scheme of finite type over a field $k$.
\begin{itemize}
\item[(i)] If $k$ is perfect, then $\Pic(G)$ is finite.
\item[(ii)] If ${\rm{char}}(k) = p > 0$, then $\Pic(G)$ is the product of a finite group and a $p$-primary group of finite exponent.
\end{itemize}
\end{theorem}

In general, affine algebraic groups over imperfect fields need not have finite Picard groups. Indeed, \cite[Prop.\,5.9]{rostrans} shows that, when $k$ is either separably closed (and imperfect) or a local function field, then this fails for almost every nontrivial $k$-form of $\Ga$. (For examples of nontrivial forms of $\Ga$ over arbitrary imperfect fields, see, for example, \cite[Ex.\,5.8]{rostrans}.)

We now turn to the proofs of the theorems above.

\begin{lemma}
\label{picinseppullback}
Let $X$ be a scheme over a field $k$, and let $k'/k$ be a field extension such that $(k')^{p^n} \subset k$. Then for all $i$, the kernel of the pullback map
\[
{\rm{H}}^i_{\et}(X, \Gm) \rightarrow {\rm{H}}^i_{\et}(X_{k'}, \Gm)
\]
is $p^n$-torsion.
\end{lemma}

\begin{proof}
Let $f\colon {\rm{Spec}}(k') \rightarrow {\rm{Spec}}(k)$ denote the obvious morphism. The composition
\begin{equation}
\label{picinseppullbackeqn1}
(\Gm)_k \rightarrow f_*((\Gm)_{k'}) \xrightarrow{[p^n]} (\Gm)_k
\end{equation}
is the $p^n$th power map, hence induces multiplication by $p^n$ on ${\rm{H}}^i(X, \Gm)$. But pushforward through a radicial map is an exact functor on the category of \'etale sheaves, hence we have a canonical identification
\[
{\rm{H}}^i(X, f_*(\Gm)) = {\rm{H}}^i(X_{k'}, \Gm)
\]
which identifies the map on cohomology induced by the first map in (\ref{picinseppullbackeqn1}) with the pullback map of the lemma. It follows that anything killed by pullback is killed by $p^n$.
\end{proof}

\begin{proof}[Proof of Theorem $\ref{pic(affine)=finitetimespprimary}$]
We are free to pass to a finite field extension. Indeed, it suffices to show that we may pass to a finite separable extension or (when $k$ is imperfect) a finite purely inseparable extension. The separable case follows from \cite[Th.\,1.13]{gjrwpic}, while the purely inseparable case, which is only relevant when $k$ is imperfect, follows from Lemma \ref{picinseppullback}. We may in particular assume that $G_{\red} \subset G$ is a smooth $k$-subgroup scheme. By \cite[Lem.\,2.2.9]{rostate}, the map $\Pic(G) \rightarrow \Pic(G_{\red})$ is an isomorphism, so we may assume that $G$ is smooth. Choose a finite extension $k'/k$ such that $G_{k'} = \coprod G_i$, where the $G_i$ are the connected components of $G_{k'}$, and such that $G_i(k') \neq \emptyset$ for each $i$. Renaming $k'$ as $k$, we have that $G_i \simeq G^0$ as $k$-schemes, hence $\Pic(G) \simeq \prod_i \Pic(G^0)$. We may therefore replace $G$ by $G^0$, so we may assume that $G$ is smooth and connected. Finally, passing to a finite purely inseparable extension (in the imperfect case) and applying \cite[Ch.\,V, Th.\,18.2(ii)]{borel}, we may assume that $G$ is unirational, so $\Pic(G)$ is finite by \cite[Th.\,1.3]{rostrans}.
\end{proof}

Let $G$ be a $k$-group scheme, and let $m, \pi_i\colon G \times G \rightarrow G$ ($i = 1, 2$) denote the multiplication and projection maps, respectively. We say that an element $\mathscr{L} \in \Pic(G)$ is {\em primitive} if $m^*(\mathscr{L}) = \pi_1^*(\mathscr{L}) + \pi_2^*(\mathscr{L})$.

\begin{lemma}
\label{actchargp}
Given an exact sequence
\begin{equation}
\label{actchargpeqn1}
1 \longrightarrow G' \xlongrightarrow{j} G \longrightarrow G'' \longrightarrow 1
\end{equation}
of smooth connected group schemes over a field $k$, and a $k$-character $\chi\colon G' \rightarrow \Gm$, consider the pushout $E$ of the sequence along $\chi$, which is a $\Gm$-torsor over $G''$. Then $E$ is a primitive torsor.
\end{lemma}

\begin{proof}
We first claim that the conjugation action of $G$ on $\widehat{G'}(k) := \Hom_{k-{\rm{gps}}}(G', \Gm)$ is trivial. Indeed, $\mathscr{D}G' \subset G'$ is a universally characteristic $k$-subgroup through which the character group factors, so in order to prove the claim we may replace the sequence (\ref{actchargpeqn1}) by the sequence obtained by forming the quotient by this subgroup, and we may assume that $G'$ is commutative. Let $T \subset G'$ be the maximal $k$-torus. Then $T$ is a characteristic $k$-subgroup of $G'$, hence normal in $G$. Further, $G'/T$ is unipotent, hence has no nontrivial $k$-characters, so $\widehat{G'}(k) \subset \widehat{T}(k)$, and it suffices to check that the $G$-conjugation action on $\widehat{T}(k)$ is trivial. But in fact we claim that $T$ is central in $G$, which is sufficient. Indeed, the automorphism functor ${\rm{Aut}}_{T/k}$ of $T$ is represented by an \'etale $k$-scheme. (This is well-known, but is easily seen by using Galois descent to pass to the case when $T \simeq \Gm^n$, in which case the automorphism functor is ${\rm{GL}}_n(\Z)$.) Conjugation induces a $k$-homomorphism $G \rightarrow {\rm{Aut}}_{T/k}$ from a connected to an \'etale $k$-group scheme, which is therefore constant. This proves that the $G$-conjugation action on $\widehat{G'}(k)$ is trivial.

Now suppose we are in the situation of the lemma. We must show that there is an isomorphism $m^*(E) \simeq \pi_1^*(E) + \pi_2^*(E)$ of $\Gm$-torsors over $G'' \times G''$. This is equivalent to constructing a map $f\colon E \times E \rightarrow E$ such that $f(t_1\cdot e_1, t_2\cdot e_2) = (t_1t_2)\cdot f(e_1, e_2)$, and such that the following diagram commutes:
\[
\begin{tikzcd}
E \times E \arrow{r}{f} \arrow{d} & E \arrow{d} \\
G'' \times G'' \arrow{r}{m} & G''.
\end{tikzcd}
\]
The torsor $E$ may be described explicitly as $E = (\Gm \times G)/i(G')$, where $i$ is the map $g' \mapsto (\chi(g'), j(g')^{-1})$ with the obvious $\Gm$-action and map to $G''$. Then one may define a map $g \colon (\Gm \times G) \times (\Gm \times G) \rightarrow \Gm \times G$ by the formula
\[
(t_1, g_1) \times (t_2, g_2) \mapsto (t_1t_2, g_1g_2).
\]
We claim that this descends to a map $f\colon E \times E \rightarrow E$. Once this is verified, the map clearly satisfies the required conditions above. A straightforward check reveals that what one needs to verify is that $\chi(g^{-1}g'g) = \chi(g')$ for all $g \in G$ and $g' \in G'$, and this is exactly the triviality of the $G$-conjugation action on $\widehat{G'}(k)$ proved above.
\end{proof}

\begin{lemma}
\label{picexactseq}
Given an exact sequence
\begin{equation}
\label{picexactseqeqn3}
1 \longrightarrow G' \longrightarrow G \xlongrightarrow{q} G'' \longrightarrow 1
\end{equation}
of smooth connected group schemes over a field $k$, then every element of $$\ker(q^*\colon \Pic(G'') \rightarrow \Pic(G))$$ is primitive.
\end{lemma}

\begin{proof}
By descent theory, $\ker(q^*)$ may be identified with the \v{C}ech cohomology group $\check{{\rm{H}}}^1(q, \Gm)$ associated to the fppf cover $q$. This group is the middle cohomology of the sequence
\begin{equation}
\label{picexactseqeqn1}
\Gamma(G, \Gm) \xlongrightarrow{p_1^* - p_2^*} \Gamma(G \times_{G''} G, \Gm) \xlongrightarrow{p_{23}^* - p_{13}^* + p_{12}^*} \Gamma(G \times_{G''} G \times_{G''} G, \Gm),
\end{equation}
where the $p_i$ and $p_{ij}$ denote the obvious projections. In particular, we obtain a surjection from the kernel of the second map above to $\ker(q^*)$. We have isomorphisms
\[
G \times_k G' \xrightarrow{\sim} G \times_{G''} G, \hspace{.3 in} (g, g') \mapsto (g, gg')
\]
\[
G \times_k G' \times_k G' \xrightarrow{\sim} G \times_{G''} G \times_{G''} G, \hspace{.3 in} (g, g'_1, g'_2) \mapsto (g, gg'_1, gg'_1g'_2).
\]
Under these isomorphisms, the second map in the sequence (\ref{picexactseqeqn1}) is transported to the map
\begin{equation}
\label{picexactseqeqn2}
\Gamma(G \times_k G', \Gm) \xlongrightarrow{f_{23}^* - f_{13}^* + f_{12}^*} \Gamma(G \times_k G' \times_k G', \Gm),
\end{equation}
where the $f_{ij}$ are defined by the following formulas:
\[
(g, g'_1, g'_2) \mapsto 
\begin{cases}
(gg'_1, g'_2), & f_{23}\\
(g, g'_1g'_2), & f_{13} \\
(g, g'_1), & f_{12}.
\end{cases}
\]
Now suppose given $h \in \Gamma(G \times_k G', \Gm)$ which lies in the kernel of the map (\ref{picexactseqeqn2}). Then restricting the equality $(h \circ f_{23})(h \circ f_{13})^{-1}(h \circ f_{12})$ to the point $(1, 1, 1)$ of $G \times_k G' \times_k G'$ shows that $h(1, 1) = 1$. It follows from the Rosenlicht Unit Theorem (more precisely, a corollary of that theorem \cite[Cor.\,1.2]{conradrosenlicht}) that $h$ is a character of $G \times_k G'$: say $h = \chi \chi'$, where $\chi$ is a character of $G$ and $\chi'$ a character of $G'$. A simple calculation shows that the fact that $h$ lies in the kernel of the map (\ref{picexactseqeqn2}) is equivalent to the identity $\chi(gg') = 1$ for all $g \in G, g' \in G'$. Thus $\chi = 1$ -- that is, $h = \chi' \in \widehat{G'}(k)$ is a character of $G'$. We thus obtain a surjection
\[
\widehat{G}'(k) \twoheadrightarrow \ker(q^*).
\]
This surjection is functorial in the exact sequence

and via the identification of $\ker(q^*)$ with $\check{{\rm{H}}}^1(q, \Gm)$, this surjection is none other than the pushout map which sends a character $\chi \colon G' \rightarrow \Gm$ to the $\Gm$-torsor over $G''$ obtained by pushing out the sequence (\ref{picexactseqeqn3}) along $\chi$. (Actually, it could also be the inverse of this map, since the identification of $\ker(q^*)$ with $\check{{\rm{H}}}^1(q, \Gm)$ is only well-defined up to a universal choice of sign.) It follows that any element of $\ker(q^*)$ is obtained by such a pushout, and Lemma \ref{actchargp} implies that such a pushout is primitive.
\end{proof}

\begin{lemma}
\label{amplepullback}
Suppose given an exact sequence
\[
1 \longrightarrow H \longrightarrow G \xlongrightarrow{q} A \longrightarrow 1
\]
of smooth connected group schemes over a field $k$, with $A$ a nonzero abelian variety. Then for any ample $\mathscr{L} \in \Pic(A)$, the pullback $q^*(\mathscr{L}) \in \Pic(G)$ is not torsion.
\end{lemma}

\begin{proof}
Replacing $\mathscr{L}$ with a power of itself, it suffices to show that $q^*(\mathscr{L}) \neq 0$. But if it were $0$, then Lemma \ref{picexactseq} would imply that $\mathscr{L}$ is primitive. In particular, the homomorphism $\phi_{\mathscr{L}}\colon A \rightarrow A$ defined by the formula $a \mapsto t_a^*(\mathscr{L}) \otimes \mathscr{L}^{-1}$ -- where $t_a$ denotes translation by $a$ -- is the $0$ map. But by \cite[Ch,.\,II, \S6, p.\,60, Application I]{mumford}, the ampleness of $\mathscr{L}$ is equivalent to the map $\phi_{\mathscr{L}}$ being an isogeny, and this yields a contradiction because $A$ is nonzero.
\end{proof}

\begin{lemma}
\label{isogabvar}
Let $k$ be a field, and suppose given an isogeny $\phi\colon G_1 \rightarrow G_2$ of finite type $k$-group schemes (that is, a faithfully flat $k$-homomorphism with finite kernel). If, in the diagram below, the rows are exact sequences of finite type $k$-group schemes with $L_i$ affine and $A_i$ abelian varieties ($i = 1, 2$), and $L_1$ smooth and connected, then the diagram of solid arrows extends uniquely to a commutative diagram with dotted arrows, and the map $f$ is an isogeny.
\[
\begin{tikzcd}
1 \arrow{r} & L_1 \arrow{r} \arrow[d, dashrightarrow] & G_1 \arrow{d}{\phi} \arrow{r}{\pi} & A_1 \arrow{r} \arrow[d, dashrightarrow, "f"] & 1 \\
1 \arrow{r} & L_2 \arrow{r} & G_2 \arrow{r} & A_2 \arrow{r} & 1.
\end{tikzcd}
\]
\end{lemma}

\begin{proof}
That the diagram extends uniquely, if it extends at all, is clear. That it extends is equivalent to the claim that the composition $L_1 \hookrightarrow G_1 \xrightarrow{\phi} G_2 \rightarrow A_2$ is trivial, and this follows from the fact that any $k$-homomorphism from a smooth connected affine $k$-group to an abelian variety is trivial. Finally, it remains to check that $f$ is an isogeny. It is surjective because $\phi$ is, so if it is not an isogeny then its kernel contains a nonzero abelian variety $B \subset \ker(f)$. Then $\pi^{-1}(B) \rightarrow L_2$ has finite kernel $F$, hence letting $G' := \pi^{-1}(B)/F$ and $L'$ the image of $\pi^{-1}(B)$ in $L_2$, we obtain a $k$-isomorphism $\psi\colon G' \xrightarrow{\sim} L'$ with the former an extension of a nonzero abelian variety $C$ by a finite type affine $k$-group $M$, and the latter an affine $k$-group of finite type. Then $G'/M \simeq L'/\psi(M)$, with the former a nonzero abelian variety and the latter affine. This is a contradiction, so $f$ is an isogeny.
\end{proof}

\begin{proof}[Proof of Theorem $\ref{torsionpic=affine}$]
Theorem \ref{pic(affine)=finitetimespprimary} implies that affine algebraic groups have torsion Picard groups, so it only remains to show that finite type $k$-groups which are not affine have non-torsion elements in their Picard groups. Because $G$ is affine if and only if $G^0$ is, we may assume that $G$ is connected. 

First we treat the case in which $k$ is perfect. Then $G_{\red} \subset G$ is a smooth connected $k$-subgroup scheme, and $G$ is affine precisely when $G_{\red}$ is, so we may assume that $G$ is smooth and connected. Chevalley's Theorem (which again uses the perfection of $k$) then furnishes an exact sequence
\[
1 \longrightarrow L \longrightarrow G \longrightarrow A \longrightarrow 1
\]
with $L$ a smooth connected affine $k$-group scheme, and $A$ an abelian variety. Because $G$ is not affine by assumption, $A$ is nonzero. Lemma \ref{amplepullback} implies that any ample line bundle on $A$ pulls back to a non-torsion element of $\Pic(G)$.

Next we treat the case in which ${\rm{char}}(k) = p > 0$. By \cite[VII$_{\rm{A}}$, Prop.\,8.3]{sga3}, there is a normal infinitesimal $k$-subgroup scheme $I \trianglelefteq G$ such that $H := G/I$ is smooth (and necessarily connected, because $G$ is). Because $G$ is not affine, neither is $H$. By Chevalley's Theorem, $H_{k_{\perf}}$ can be written as an extension of an abelian variety by a smooth connected affine $k_{\perf}$-group scheme. It follows that there is such an extension over $k^{1/p^n}$ for some integer $n \geq 0$. Extending scalars via the $p^n$th-power isomorphism $k^{1/p^n} \xrightarrow{\sim} k$, we obtain an exact sequence
\[
1 \longrightarrow L \longrightarrow H^{(p^n)} \longrightarrow A \longrightarrow 1
\]
with $L$ a smooth connected affine $k$-group and $A/k$ a nonzero abelian variety (because $H$, hence also $H^{(p^n)}$, is not affine). Denote by $q$ the composite map
\[
G \rightarrow H \xrightarrow{F_{H/k}^{(p^n)}} H^{(p^n)} \rightarrow A,
\]
where the map in middle is the $n$-fold relative Frobenius map. Let $\mathscr{L}$ be an ample line bundle on $A$. We claim that the pullback $q^*(\mathscr{L}) \in \Pic(G)$ is not torsion.

It suffices to check this after extending scalars to $\overline{k}$, hence we are now free to assume that $k = \overline{k}$. Then $G_{\red} \subset G$ is a smooth connected $k$-subgroup scheme, hence by Chevalley there is an exact sequence
\[
1 \longrightarrow M \longrightarrow G_{\red} \xlongrightarrow{g} B \longrightarrow 1
\]
with $M$ a smooth connected affine $k$-group and $B/k$ an abelian variety. The composition $G_{\red} \subset G \rightarrow H \xrightarrow{F_{H/k}^{(p^n)}} H^{(p^n)}$ is an isogeny: each map in this sequence is set-theoretically surjective, and $G_{\red}$ and $H^{(p^n)}$ are smooth, so the map is fppf, and its kernel is finite because that is true of each map in the above composition. Lemma \ref{isogabvar} therefore implies that we have a unique commutative diagram with exact rows
\[
\begin{tikzcd}
1 \arrow{r} & M \arrow{r} \arrow{d} & G_{\red} \arrow{d} \arrow{r}{g} & B \arrow{r} \arrow{d}{f} & 1 \\
1 \arrow{r} & L \arrow{r} & H^{(p^n)} \arrow{r} & A \arrow{r} & 1,
\end{tikzcd}
\]
and that $f$ is an isogeny. In particular, $f^*(\mathscr{L})$ is still ample, hence by Lemma \ref{amplepullback}, $g^*f^*(\mathscr{L}) \in \Pic(G_{\red})$ is not torsion. The commutative diagram
\[
\begin{tikzcd}
G_{\red} \arrow{r} \arrow{d}{g} & G \arrow{d}{q} \\
B \arrow{r}{f} & A
\end{tikzcd}
\]
then shows that $q^*(\mathscr{L})$ is not torsion either.
\end{proof}

\noindent {\bf Acknowledgements.} It is a pleasure to thank Nguy\~{\^{e}}n Qu\'{\^{o}}c Th\'{\u{a}}ng for pointing out to me the reference \cite{gjrwpic}, and in particular Theorem 1.13 of that paper.

\noindent \address
\vspace{.3 in}

\noindent \email


\begin{thebibliography}{ram}

\bibitem[Bor]{borel} Armand Borel, {\em Linear Algebraic Groups}, 2nd edition, Springer GTM 126, 1991.

\bibitem[Con]{conradrosenlicht} Brian Conrad, {\em Units On Product Varieties}, available at {\tt{http://math.stanford.edu/~conrad/papers/unitthm.pdf}}.

\bibitem[SGA3]{sga3} Michel Demazure, Alexander Grothendieck, {\em Sch\'emas en Groupes}
I, II, III, Lecture Notes in Math {\bf 151, 152, 153}, Springer-Verlag, New York (1970).

\bibitem[GJRW]{gjrwpic} Robert Guralnick, David B.\,Jaffe, Wayne Raskind, and Roger Wiegand, {\em On the Picard Group: Torsion and the Kernel Induced by a Faithfully Flat Map}, Journal of Algebra {\bf 183}, 420--455 (1996).

\bibitem[Mum]{mumford} David Mumford, {\em Abelian Varieties}, Oxford University Press (1970), London.

\bibitem[Ros1]{rostate} Zev Rosengarten, {\em Tate Duality In Positive Dimension Over Function Fields}, to appear in Memoirs of the American Mathematical Society.

\bibitem[Ros2]{rostrans} Zev Rosengarten, {\em Translation-Invariant Line Bundles On Linear Algebraic Groups}, The Journal of Algebraic Geometry, published electronically March 16, 2020.

\end{thebibliography}
\end{document}